\documentclass[11pt]{article}
\usepackage[margin=1in]{geometry}

\usepackage[pdftex]{graphicx}
\graphicspath{{../pdf/}{../jpeg/}}
\DeclareGraphicsExtensions{.pdf,.jpeg,.png}
\usepackage{amssymb,amsmath,url}
\usepackage{bm} 
\usepackage{multirow}
\usepackage{booktabs}	
\usepackage{epstopdf}
\usepackage{color}
\usepackage{subfigure}
\usepackage{bigfoot}		
\usepackage{algpseudocode}
\usepackage{pifont}
\usepackage{tikz}
\usepackage{color}
\usepackage{cite}
\usepackage{lipsum}
\usepackage{algorithm} 
\usepackage[pdftex]{graphicx}

\algrenewcommand\algorithmicrequire{\textbf{Input:}}
\algrenewcommand\algorithmicensure{\textbf{Output:}}
\algrenewcommand\algorithmicforall{\textbf{For}}

\newtheorem{theorem}{Theorem}
\newtheorem{lemma}{Lemma}

\allowdisplaybreaks
\newtheorem{definition}{Definition} 
\newtheorem{proof}{Proof}

\newcommand{\R}{\mathbb{R}}

\newcommand*{\QEDA}{\hfill\ensuremath{\blacksquare}}

\title{\LARGE
	Actuator Placement for Optimizing Network Performance under Controllability Constraints
}
\author{Baiwei Guo\thanks{B. Guo, O. Karaca, M. Kamgarpour are with the Automatic Control Laboratory, Department of Information Technology and Electrical Engineering, ETH Z\"{u}rich, Switzerland. e-mails: {\tt \{bguo, okaraca, mkamgar\}@ethz.ch}.} \and Orcun Karaca\footnotemark[1] \and Tyler H. Summers\thanks{ T. Summers is with the Department of Mechanical Engineering, University of Texas at Dallas, Richardson, TX, USA. email:  {\tt tyler.summers@utdallas.edu}.} \and Maryam Kamgarpour\footnotemark[1]
}
\allowdisplaybreaks
\begin{document}
\maketitle \let\thefootnote\relax\footnotetext{The work of M. Kamgarpour and O. Karaca was gratefully funded by the European Union ERC Starting Grant CONENE.}\\
\let\thefootnote\relax\footnotetext{The work of T. Summers was sponsored by the Army Research Office and was accomplished under Grant Number: W911NF-17-1-0058.}
\begin{abstract}\noindent With the rising importance of large-scale network control, the problem of actuator placement has received increasing attention. Our goal in this paper is to find a set of actuators minimizing the metric that measures the average energy consumption of the control inputs while ensuring structural controllability of the network. As this problem is intractable, greedy algorithm can be used to obtain an approximate solution. To provide a performance guarantee for this approach, we first define the submodularity ratio for the metric under consideration and then reformulate the structural controllability constraint as a matroid constraint. This shows that the problem under study can be characterized by a matroid optimization involving a weakly submodular objective function. Then, we derive a novel performance guarantee for the greedy algorithm applied to this class of optimization problems. Finally, we show that the matroid feasibility check for the greedy algorithm can be cast as a maximum matching problem in a certain auxiliary bipartite graph related to the network graph.
\end{abstract}
\section{Introduction}
Actuator placement is the problem of finding a subset from a finite set of possible placements for actuators
to optimize a desired network performance metric. With the increased importance of large-scale network control, such as those arising in power grids and transportation systems, there has been a surge of interest to study the problem of actuator placement. Past works have discussed several controllability-based performance metrics and derived properties of the resulting optimization problems~\cite{MULLER1972237,Redmond1996}.

The problem of actuator placement is in general NP-hard~\cite{Ali2017MinimalReachablity}. Hence, earlier studies have adopted the greedy algorithm to derive an approximate solution\cite{SubmodularityandControl}. Under a submodular metric and a cardinality constraint on the number of actuators allowed, the greedy algorithm is shown to enjoy a provable suboptimality guarantee~\cite{nemhauser1978analysis}.
However, some metrics, such as the minimum eigenvalue of the controllability Gramian and the trace of the inverse of the Gramian, do not exhibit submodularity~\cite{summers2017performance,tyler2018correction}. To alleviate this issue, the notion of submodularity has been extended to weak submodularity using submodularity ratio quantifying how close a function is to being submodular~\cite{Das:2011:SMS:3104482.3104615,bian2017guarantees}. Given this ratio, it is possible to derive a performance guarantee for the greedy algorithm applied to a larger class of performance metrics~\cite{summers2017performance}.

Nonetheless, the guarantees above are restricted to optimization problems subject to simple cardinality constraints. Given a cardinality constraint, the resulting actuator set might not render the system controllable. 
To address this issue, we need to include controllability constraints in the optimization problem. However, to the best of our knowledge, there is no approach to ensure feasibility of the iterates of the greedy algorithm applied to this problem. 
On the other hand, structural controllability constraints have been well-studied since this controllability concept is based only on the graphical interconnection structure of the dynamical system under consideration~\cite{SturcturalcontrollabilityTai,liu2011controllability}. Structurally controllable systems are those controllable after a slight perturbation of the system parameters corresponding to the fixed set of edges in the underlying network graph. The authors in~\cite{clark2012leader} have studied a leader selection problem to obtain a structurally controllable system while minimizing a submodular objective function. The structural controllability constraints arising in the leader selection problem are proven to be equivalent to so-called matroid constraints~\cite{clark2012leader}. However, the leader selection problem is different from the actuator placement problem since the former selects a set of states to act as external control inputs controlling the dynamics of the remaining states.
To this end, our first goal is to show that the actuator placement problems under structural controllability constraints can also be cast as a matroid optimization.

To obtain a performance guarantee for the greedy algorithm applied to a matroid optimization problem, the authors in~\cite{fisher1978analysis} have studied the case with a submodular objective function. As an extension, the authors in \cite{pmlr-v80-chen18b} have considered weakly submodular objective functions. This setting captures the actuator placement problems under structural controllability constraints. However, the performance guarantees in~\cite{pmlr-v80-chen18b} are restricted to the residual random greedy algorithm. To the best of our knowledge, there is no guarantee obtained for the greedy algorithm applied to a matroid optimization if the objective is weakly submodular. Therefore, our second goal is to obtain a performance guarantee for the greedy algorithm applied to this problem.

Our contributions are as follows. 
First, we show that the actuator placement problem optimizing a nonsubmodular controllability metric under structural controllability constraints can be cast as a matroid optimization, see Theorem~\ref{thm: matroid}. Second, we introduce a novel notion of submodularity ratio and show that the metric under consideration satisfies this property. Third, using the introduced notion of submodularity ratio, we bound the worst-case performance of the greedy algorithm applied to the class of optimization problems with weakly submodular objective functions and matroid constraints, see Theorem \ref{thm:upperlimitFormatroidconstraints}. This enables us to bound the greedy algorithm's performance on the actuator placement problem under structural controllability constraints. Finally, we show that the matroid feasibility check for the greedy algorithm is equivalent to a maximum matching problem in a bipartite graph related to the network graph, see Theorem~\ref{thm: feasibility}.

The remainder of this paper is organized as follows. In Section \ref{sec:mech}, we introduce the system model and the problem formulation. In Section~\ref{sec:problemstructure}, we study the properties of the controllability metric and reformulate the structural controllability constraints such that the feasible region characterizes a matroid. Section~\ref{sec: performance guarantee} obtains a performance guarantee for the greedy algorithm applied to our problem class. In Section~\ref{sec:num}, we discuss the implementation of the greedy algorithm for our problem and include numerical case studies. 

\section{Problem Formulation}\label{sec:mech}

In this section, we introduce the system model, the concept of structural controllability and a metric measuring the average energy required to control the system. Our main task is formulated as minimization of this energy-based metric through an actuator placement that renders the system structurally controllable.

\subsection{System Model}
Consider a linear control system with state vector $x\in \R^n$. To each state variable $x_i\in \R$, we associate a node $v_i$. A control input $u_i\in \R$ can be exerted at each node $v_i\in V:= \{v_1, \ldots, v_n\}$. Given a set $S\subset V$ chosen as the actuator set, dynamics can be written as
\begin{equation}
\begin{aligned}
\label{eq: systemmodel}
\dot{x} = Ax + B(S)u.
\end{aligned}
\end{equation}
We let $B(S) := \text{diag}(\bm {1}(S))$, where $\bm {1}(S)$ denotes a vector of size $n$ whose $i$th element is $1$ if $v_i$ belongs to $S$ and $0$ otherwise.  We let $G=(V,E)$ denote a digraph with nodes $V$ and edges $E$, where the edge $(v_i,v_j)\in E$ if and only if $(A)_{ij}$ is not $0$. 
\subsection{Problem Statement}
The pair $(A,B(S))$ is called controllable if the states $x$ can be steered arbitrarily in $\R^n$ in any given finite time. It is known that controllability can be verified by the rank of the controllability matrix $P=\begin{bmatrix}
   B(S) & AB(S) & \ldots & A^{n-1}B(S)
\end{bmatrix}\in \R^{n\times n^2}$. If $(A,B(S))$ is not controllable, it might be possible to slightly perturb the entries in $A$ and $B(S)$ to ensure controllability \cite{SturcturalcontrollabilityTai}. As the entries of the matrix $A$ are generally known with some errors, the notion of controllability is not robust. A more appropriate notion is structural controllability defined below.

\begin{definition}
\label{def: structuralcontrollability}
$(A,B)$ and $(\hat{A},\hat{B})$ with $A,B,\hat{A},\hat{B}$ $\in \R^{n\times n}$ are said to have the same structure if matrices $[A\text{ }B]$ and $[\hat{A} \text{ }\hat{B}]$ have fixed zeros at the same entries. Given $S\subset V$, $(A,B(S))$ is structurally controllable if there exists a controllable pair $(\hat{A},\hat{B})$ having the same structure as $(A,B(S))$. We say $S$ is a \textit{capable actuator set} if $(A,B(S))$ is structurally controllable.
\end{definition}

Even if a system is controllable, an unacceptably large amount of energy might be needed to reach a desired state. Hence, it is crucial to minimize this energy consumption. The minimum energy required to steer the system from $x_0\in\R^n$ at $t=0$ to zero at $t=T$ is given by $x_0^\top W^{-1}_T(S)x_0,$ where $W_T(S)=\int_{0}^{T} e^{A\tau}B(S) B^\top (S) e^{A^\top\tau} d\tau$ is the controllability Gramian. To obtain an expression independent of the initial state $x_0$, we can calculate the average energy required over the unit sphere, $||x_0||_2=1$, as $\text{tr}(W_T^{-1}(S)).$
This expression is well-defined only when the set $S$ renders the system controllable. Inspired by \cite{EffortBounds}, we introduce a small positive number $\epsilon\in \R_+$ and propose the following metric
\begin{equation}
\label{eq:F(S)}
F(S) =  \text{tr}((W_T(S)+\epsilon I)^{-1}), \text{ }\forall S\subset V.
\end{equation}

To make a system easier to control, we seek a set $S\subset V$ minimizing the metric above. Since in a large-scale network, the number of actuators allowed is in general limited, we consider a cardinality bound of $K\in \mathbb{Z}_+$ on the number of actuators allowed. Additionally, we need a controllability constraint to ensure that the actuator set is capable. Therefore, our main problem is formulated as
\begin{equation}
\begin{aligned}
& \min_{S}
& & F(S)
\\ 
&\ \mathrm{s.t.} & & |S|\leq K \text{ and $S$ is a capable actuator set.}
\end{aligned}
\label{eq:mainproblem2}
\end{equation}
Notice that $K$ needs to be large enough to ensure feasibility of the problem above. In Section \ref{sec:feasibilitycheck}, this issue is addressed by deriving a lower bound on $K$. Though we do not provide any analysis on the complexity of Problem (\ref{eq:mainproblem2}), to the best of our knowledge, no computationally feasible method of finding the exact optimum has ever been proposed. A heuristic method, called the greedy algorithm, has been broadly adopted to derive an approximate solution. This algorithm starts from the empty set and iteratively adds the element with the largest marginal gain. 
In the next two sections, our goal is to derive a performance guarantee for the greedy algorithm applied to Problem~\eqref{eq:mainproblem2}.

\section{Characterization of Problem Structure}
\label{sec:problemstructure}
If the set function $-F$ is submodular and the constraints form a matroid, it is known that one can derive a performance guarantee for the greedy algorithm \cite{fisher1978analysis}. However, the work in \cite{summers2017performance} shows that the set function $-F$ is not submodular. Moreover, there is no work characterizing the constraints found in Problem~\eqref{eq:mainproblem2}. 
In the following, we show that \textit{a)} one can analyze the submodularity ratio of the nonsubmodular function $-F$ and \textit{b)} the constraints in Problem~\eqref{eq:mainproblem2} form a matroid via a reformulation.

\subsection{Properties of the Objective}
We say that a set function $f:2^V\rightarrow \R$ is (strictly) increasing if $f(S_1)\leq$($<$)$ f(S_2)$ for any $S_1\subsetneqq S_2\subset V$. Similarly, we say that $f$ is (strictly) decreasing if $-f$ is (strictly) increasing. Intuitively, with more input nodes, the System~(\ref{eq: systemmodel}) would be easier to control, and thus the metric $F$ in~(\ref{eq:F(S)}) would be smaller. This intuition can be readily verified as follows.
\begin{lemma}
The metric $F$ in~\eqref{eq:F(S)} is strictly decreasing.
\end{lemma}
\begin{proof}
For any $S \subset V$ and any $\omega\in V\setminus S$, let $H(z)=(W_T(S)+zW_T(\{\omega\})+\epsilon I)^{-1}$. Notice that
$$\text{tr}(H(1)) = \text{tr}((W_T(S\cup \{\omega\})+\epsilon I)^{-1}) = F(S\cup \{\omega\}),$$
since $W_T(S)+W_T(\{\omega\})= W_T(S\cup \{\omega\}).$
Then, we obtain
$$\frac{d(\text{tr}(H(z)))}{dz}=-\text{tr}(H(z)W_T(\{\omega\})H(z))< 0, \text{ }\forall z\in(0,1).$$
The inequality above holds because $H(z)$ is symmetric and $W_T(\{\omega\})$ is positive semidefinite.
Thus, $\text{tr}(H(1))-\text{tr}(H(0))< 0$.\hfill\QEDA
\end{proof}

Next, we introduce the notion of submodularity ratio. For the following, $\rho_{U}(S):=f(S\cup U)-f(S)$, for all $S,U\subset V$.
\begin{definition}
For a nondecreasing function $f \colon 2^V \to \R$, submodularity ratio is the largest $\gamma \in \R_+$ such that
\begin{equation}
\begin{aligned}
\label{eq:definition for submodularity ratio}
\gamma\rho_\omega(S\cup U)\leq  \rho_\omega(S),\text{ } \forall S,U,\{\omega\} \subset V. 
\end{aligned}
\end{equation}
A set function $f$ with submodularity ratio $\gamma$ is called $\gamma$-submodular. A $\gamma$-submodular set function is said to be submodular if $\gamma=1$ and weakly submodular if $0<\gamma<1$. 
\label{def:submodularityratio}
\end{definition}

Clearly, for any nondecreasing set function, $\gamma\in[0,1]$. Since the metric $F$ is decreasing, we instead consider the submodularity ratio $\gamma$ of $-F$. Due to the strict monotonicity of the metric $F$, we have $\gamma > 0$. Thus, $-F$ is weakly submodular. {In Appendix, we connect Definition~\ref{def:submodularityratio} with another existing notion of submodularity ratio, and discuss the need to introduce this notion as per Definition~\ref{def:submodularityratio}. }

\subsection{Reformulation of the Constraint Set}
Since $-F$ is strictly decreasing, the optimal solution to Problem (\ref{eq:mainproblem2}), denoted as $S^*$, satisfies $|S^*|=K$. As a result, we can define $\mathcal{C}_K=\{S\subset V\text{ }|\text{ }|S| = K\text{ and}$ $ S\text{ is}$ $\text{a capable actuator set}\}$ and rewrite Problem (\ref{eq:mainproblem2}) as the minimization of $F$ over the set $\mathcal{C}_K$. For the greedy algorithm, at any iteration $i\leq K$, let $S^i$ denote the actuator set obtained at the $i$th iteration. This selection $S^i$ has to be a subset of some set in $\mathcal{C}_K$. Otherwise, the greedy solution $S^K$ may not be feasible in $\mathcal{C}_K$. Keeping this in mind, we define $\mathcal{\tilde{C}}_K=\{\Omega\text{ }|\text{ }\exists \text{ } S \in \mathcal{C}_K \text{ such that } \Omega \subset S \}$ and reformulate (\ref{eq:mainproblem2}) as   
\begin{equation}
\begin{aligned}
& \min_S
& & F(S)
\\ 
&\ \mathrm{s.t.} & & S\in \mathcal{\tilde{C}}_K.
\end{aligned}
\label{eq:mainproblem}
\end{equation}
The strict monotonicity of $F$ again ensures that the optimal solution to (\ref{eq:mainproblem}) coincides with the one to (\ref{eq:mainproblem2}). As such, for the rest of the paper, we consider solving Problem (\ref{eq:mainproblem}) as an equivalent characterization of Problem (\ref{eq:mainproblem2}).

Next, we show that the feasible region of Problem (\ref{eq:mainproblem}) has a special structure that helps us bound the worst-case performance of the greedy algorithm. To this end, we first bring in the definition of a matroid. We then show that the constraint set $\tilde{\mathcal{C}}_K$ characterizes a matroid.
\begin{definition}
A matroid $\mathcal{M}$ is an ordered pair $(V,\mathcal{F})$ consisting of a ground set $V$ and a collection $\mathcal{F}$ of subsets of $V$ which satisfies (i) $\emptyset \in \mathcal{F}$, (ii) if $S \in \mathcal{F}$ and $S'\subset S$, then $S'\in \mathcal{F}$, (iii) if $S_1$,$S_2\in \mathcal{F}$ and $|S_1|<|S_2|$, there exists $\omega \in S_2\setminus S_1$ such that $\{\omega\}\cup S_1\in \mathcal{F}$. Every set in $\mathcal{F}$ is called \textit{independent}.
\end{definition}
\begin{theorem}
\label{thm: matroid}
$\mathcal{M}=(V,\tilde{\mathcal{C}}_K)$ is a matroid. 
\end{theorem}

\begin{proof}
To prove this theorem, we show that given an actuator set $S$, structural controllability of $(A,B(S))$ can be equivalently formulated as structural controllability of the system with the set $S$ as a leader selection. The concept of leader selection was studied in~\cite{patterson2010leader}. We define $N= V\setminus S$ and partition accordingly the state vector $x$ into $x_S$ and $x_N$. The dynamics of the system can equivalently be written as 
\begin{equation}
\label{eq: statepartition}
\begin{bmatrix}
    \dot{x}_{N} \\ \dot{x}_{S}
\end{bmatrix} = \begin{bmatrix}
 {A}_{NN} & {A}_{NS}  \\
 {A}_{SN} & {A}_{SS}
\end{bmatrix} \begin{bmatrix}
    x_{N} \\ x_{S}
\end{bmatrix} + \begin{bmatrix}
   0 & 0\\ 0 & I_{|S|}
\end{bmatrix}u,
\end{equation}
where $I_{|S|}\in \R^{|S|\times|S|}$ is the identity matrix. Consider the dynamics of $x_N$ where the inputs are given by $x_S$, that is,
$$\dot{x}_N = A_{NN}x_N+A_{NS}x_S.$$We say that $S$ is a leader selection that achieves structural controllability if  $(A_{NN},A_{NS})$ is structurally controllable.

To start with, from Definition \ref{def: structuralcontrollability}, we know that $S$ is a capable actuator set if and only if there exists a pair $(\hat{A},\hat{B})$ with the same structure as $(A,B(S))$ such that the controllability matrix $P\in \R^{n\times n^2}$ given by
$$P = \begin{bmatrix}
   0 & 0 & 0 &\hat{A}_{NS} & 0 & \hat{A}_{NN}\hat{A}_{NS}+\hat{A}_{NS}\hat{A}_{SS} &\cdots\\
   0 & I_{|S|} & 0 & \hat{A}_{SS} & 0 & \hat{A}_{SN}\hat{A}_{NS}+\hat{A}^2_{SS} & \cdots
   \end{bmatrix},$$
   has full rank.
Next, we claim that $P$ has full rank if and only if the following matrix $\tilde{P_1}\in \R^{|N|\times n^2}$ has full rank, 
$$\tilde{P_1} = \begin{bmatrix}
   0 & 0 & 0 &\hat{A}_{NS} & \cdots &0 & \hat{A}^{j-1}_{NN}\hat{A}_{NS} & \cdots 
   \end{bmatrix}.$$
To see this, we notice that $P$ has full rank if and only if the submatrix $P_1\in \R^{|N|\times n^2}$ containing the first $|N|$ rows of $P$ has full rank. One can then easily show that there exists an upper triangular matrix $U\in \R^{n^2\times n^2}$ with unit diagonal entries such that $\tilde{P_1}= {P_1}U$.  Since $U$ is invertible, $\tilde{P_1}$ and $P_1$ have the same rank.

Then, we further claim that $\tilde{P_1}$ has full rank if and only if the following matrix $\bar{P_1}$ has full rank $$\bar{P_1} = \begin{bmatrix}
\hat{A}_{NS} & \hat{A}_{NN}\hat{A}_{NS} & \cdots & \hat{A}^{|N|-1}_{NN}\hat{A}_{NS}\end{bmatrix}.$$
Considering $|S|>0$ and thus $|N|-1\leq n-2$, for any $i>{|N|-1}$, $\hat{A}^{i}_{NN}\hat{A}_{NS}$ is in the span of the matrices $\hat{A}^{j}_{NN}\hat{A}_{NS}$, $j=\{0,1,\ldots,|N|-1\}$ by Cayley-Hamilton theorem. Hence, $\bar{P_1}$ has the same rank as $\tilde{P_1}$. This proves the claim.

In summary, $P$ has full rank if and only if $\bar{P_1}$ has full rank. By the definition of $\bar{P_1}$, $\bar{P_1}$ being full rank is equivalent to controllability of $(\hat{A}_{NN},\hat{A}_{NS})$. Hence, structural controllability of $(A,B(S))$ is equivalent to structural controllability of $({A}_{NN},{A}_{NS})$. We define $\mathcal{L}_K = \{S\text{ }|\text{ $|S| = K$ and}$ $({A}_{NN},{A}_{NS})\ is\ structurally\ controllable\}$ and conclude that $\mathcal{L}_K = \mathcal{C}_K$. The set collection $\mathcal{L}_K$ consists of all the leader selections that achieve structural controllability. From \cite{clark2012leader}, we have that $(V,\mathcal{\tilde{L}}_K)$, where $\mathcal{\tilde{L}}_K:=\{\Omega\text{ }|\text{ } \exists \text{ } S\in \mathcal{L}_K$ $ \text{such that }$ $ \Omega\subset S \}$, is a matroid. Therefore, $(V,\mathcal{\tilde{C}}_K)$ is also a matroid.\hfill\QEDA
\end{proof}

The concept of leader selection is studied in the realm of social networks \cite{patterson2010leader}. From the above proof, 
we establish the equivalence between finding a capable actuator set and finding a leader selection achieving structural controllability.

\section{Performance Guarantee}
\label{sec: performance guarantee}
By considering $f = -F$ as the objective function, Problem (\ref{eq:mainproblem}) falls into the following class of optimization problems:
\begin{equation}
\begin{aligned}
& \max_{S\subset V}
& & f(S), \text{ nondecreasing and $\gamma$-submodular}
\\ 
& \ \mathrm{s.t.} & & S\in \mathcal{F}, \text{ where $\mathcal{M}=(V,\mathcal{F})$ is a matroid},
\end{aligned}
\label{eq:matroidconstrained optimization}
\end{equation}
where the cardinality of the largest set in $\mathcal{F}$ is assumed to be $K$. Our goal is to derive a performance guarantee for the greedy algorithm applied to Problem (\ref{eq:matroidconstrained optimization}).

The greedy algorithm is presented in Algorithm \ref{alg:ALG4} and described as follows. At the $t$th iteration of the algorithm, the actuator set returned by the previous iteration is denoted as $S^{t-1}$. We check the feasibility of the node with the largest marginal gain in $V\setminus S^{t-1}$. If the actuator set obtained by adding this node does not belong to $\tilde{\mathcal{C}}_K$, we exclude the node from consideration. Among the remaining nodes, we repeatedly check the feasibility of the node with the largest marginal gain until a feasible node $v^\text{G}_t$ is found. Then $S^t=\{v^\text{G}_t\}\cup S^{t-1}$ is the actuator set returned by the $t$th iteration. The final actuator set is $S^\text{G}:=S^K$. The feasibility check ensures that $S^t\in \mathcal{F}$. Hence, $S^\text{G}$ belongs to $\mathcal{F}$.

We define $U^{-1}=\emptyset$ and also use notations $U^t\subset V$ for $0\leq t\leq K-1$ such that $U^t\subset V$ expands at the $(t+1)$th iteration and includes all the nodes having been considered by the feasibility check before $v^\text{G}_{t+1}$ when the algorithm terminates. We also define $U^K = V$ and $\rho_t = f(S^{t+1})-f(S^t)$. 

\begin{algorithm}[H]
        \caption{Greedy Algorithm on Matroid Optimization}
        \label{alg:ALG4}
        \begin{algorithmic}
        \Require set function $f$, ground set $V$ and matroid $(V,\mathcal{F})$
        \Ensure actuator set $S^\text{G}$
        \Function {GreedyonMatroid}{$f, V,\mathcal{F}$}
        \State {$S^0 = \emptyset$, $U^0= \emptyset$, $t=1$}
        \While{$U^{t-1}\neq V$}
            \State ${i^*(t)}= \arg\max_{i\in V\setminus U^{t-1}}\rho_i(S^{t-1})$
            \If {$S^{t-1}\cup\{i^*(t)\}\notin\mathcal{F}$}
        \State $U^{t-1}\gets U^{t-1}\cup\{i^*(t)\}$
        \Else
        \State $\rho_{t-1} \gets\rho_{i^*(t)}(S^{t-1})$ and $v^\text{G}_t = i^*(t)$
        \State $S^{t} \gets S^{t-1}\cup\{v^\text{G}_t\}$ and $U^{t}\gets U^{t-1}\cup\{v^\text{G}_t\}$
        \State $t\gets t+1$
    \EndIf    
            \EndWhile
            \State $S^\text{G} \gets S^{t-1}$
        \EndFunction
        \end{algorithmic}
    \end{algorithm}

Our main result is as follows.
\begin{theorem}
If Algorithm \ref{alg:ALG4} is applied to Problem (\ref{eq:matroidconstrained optimization}), then 
\begin{equation}
\begin{aligned}
\frac{f(S^\text{G})-f(\emptyset)}{f(S^*)-f(\emptyset)}\geq\frac{\gamma^3}{\gamma^3+1}.
\end{aligned}
\label{eq: performanceguaranteeMatroidOpt}
\end{equation}

\label{thm:upperlimitFormatroidconstraints}
\end{theorem}

The idea of the proof is based on the work of \cite{fisher1978analysis}, which derives the performance guarantee for matroid optimization problems featuring submodular objectives. 

To assess suboptimality of the greedy selection $S^\text{G}$, we aim to find an upper bound for $f(S^*)-f(S^\text{G})$. We denote $S^*=\{v^*_1,\ldots, v^*_K\}$ and notice
\begin{equation}
\begin{aligned}
f(S^*)-f(S^\text{G}) &\leq f(S^*\cup S^\text{G})-f(S^\text{G})\\
 &= \sum^{K}_{k=1} \rho_{v^*_k}(\{v^*_1,\ldots, v^*_{k-1}\}\cup S^\text{G})\\
 &\leq \gamma^{-1}\sum_{j\in S^*\setminus S^\text{G}}\rho_j(S^\text{G}),
\end{aligned}
\label{eq: ideaTheorem}
\end{equation}
where the first inequality is due to the monotonicity of $f$ and the equality shows a partition of $f(S^*\cup S^\text{G})-f(S^\text{G})$. The last inequality is due to Definition \ref{def:submodularityratio}. To further bound $\sum_{j\in S^*\setminus S^\text{G}}\rho_j(S^\text{G})$, we need the following lemmas.
\begin{lemma}
Let $s_{t-1} = |S^*\cap(U^t\setminus U^{t-1})|$, then
\begin{equation}
\begin{aligned}
\sum_{j\in S^*\setminus S^\text{G}}\rho_j(S^\text{G})\leq \gamma^{-1} \sum_{t=1}^{K}\rho_{t-1}s_{t-1}.
\end{aligned}
\label{eq: lemmainequality1}
\end{equation}
\label{lmm: totalInequality}
\end{lemma}
\begin{proof}
From Definition \ref{def:submodularityratio}, we have
\begin{equation}
\begin{aligned}
\rho_j(S^\text{G})  & \leq \gamma^{-1}\rho_j(S^{t-1}), \forall t\leq K\text{, } \forall j\in V.
\end{aligned}
\label{eq: submodularityratioTheorem}
\end{equation}
Since $U^{t_1}\subset U^{t_2}$ for any $t_1<t_2$, notice that
\begin{equation}
V = U^K =  \bigcup^{K}_{t=0}(U^t\setminus U^{t-1}).
\label{eq:partition}
\end{equation}
Also considering $U^{t_1}\setminus U^{t_1-1}$ and $U^{t_2}\setminus U^{t_2-1}$ are disjoint, we know (\ref{eq:partition}) constitutes a partition of $V$.
Since there is no subset of $U^0$ belonging to $\mathcal{F}$, we have $S^*\cap U^0 = \emptyset$. With the partition of $V$ as in (\ref{eq:partition}), we can partition $S^*$ as:
$$S^*=\overset{{K}}{\underset{t=1}{{\bigcup}}}(S^*\cap(U^t\setminus U^{t-1})).$$
Combining the above partition with (\ref{eq: submodularityratioTheorem}), we have
\begin{equation}
\begin{aligned}
\sum_{j\in S^*\setminus S^\text{G}}\rho_j(S^\text{G}) & \leq \sum_{j\in S^*}\rho_j(S^\text{G})
\\ 
& = \sum_{t=1}^{K}\sum_{j\in S^*\cap(U^{t}\setminus U^{t-1})}\frac{1}{\gamma}\rho_j(S^{t-1}).
\end{aligned}
\label{eq:decomposition of S^*_m}
\end{equation}
Notice that all the nodes in $U^{t-1}$ have been considered by the feasibility check before $v^\text{G}_t$. Since the greedy algorithm first checks the elements in $V\setminus U^{t-1}$ with larger marginal gains when added to $U^{t-1}$, we have that 
$$\rho_{t-1} = \max_{j\in V\setminus U^{t-1}}\rho_j(S^{t-1}).$$
Considering $V\setminus U^{t-1}= \cup_{i=t}^{K} (U^i\setminus U^{i-1})$,
for any $t'\geq t$,
\begin{equation}
\begin{aligned}
\rho_{t-1}\geq \rho_j(S^{t-1}), \forall j\in U^{t'}\setminus U^{t'-1}.
\end{aligned}
\label{eq: greedypickproperty}
\end{equation}
Thus, for any $j\in S^*\cap(U^t\setminus U^{t-1})$, we have $\rho_j(S^{t-1})\leq \rho_{t-1}$ and 
\begin{equation}
    \sum_{j\in S^*\cap(U^{t}\setminus U^{t-1})}\rho_j(S^{t-1}) \leq \rho_{t-1} s_{t-1}.
    \label{eq: backwardnumberinequality}
\end{equation}
Now combining (\ref{eq:decomposition of S^*_m}) and (\ref{eq: backwardnumberinequality}), it is straightforward that
\begin{equation*}
\sum_{j\in S^*\setminus S^\text{G}}\rho_j(S^\text{G})\leq  \sum_{t=1}^{K}\gamma^{-1}\rho_{t-1}s_{t-1}.   
\end{equation*}\hfill\QEDA
\end{proof}
\begin{lemma}
For any $t\in \{1,\ldots,K\}$, we have 
\begin{equation}
\sum_{i=1}^{t}s_{i-1} \leq t.
\label{eq: lemmainequality2}
\end{equation}
\label{lmm: secondInequality}
\end{lemma}

This has been proven by \cite{fisher1978analysis} for $\gamma=1$. Since the proof only exploits the matroid structure of $S^*$, the above lemma holds also when $\gamma\neq 1$. We include the proof for the sake of completeness.

\begin{proof}
To start with, we claim that any independent subset of $U^t$ has a cardinality no more than $t$. Otherwise, due to the matroid structure of $\mathcal{F}$, there exists $j\in U^t \setminus S^t$ $\text{such that } S^t\cup \{j\}$ is independent. Since $j\in U^t$ and $U^t = \cup^t_{i=0} (U^i\setminus U^{i-1})$ is a partition, there exists $t'\leq t$ such that $j\in U^{t'}\setminus U^{t'-1}$. Since $S^{t'}\cup\{j\}\subset S^{t}\cup\{j\}$, $S^{t'}\cup\{j\}$ is independent. By the mechanism of the greedy algorithm, we know $j$ passes the feasibility check ahead of $v^\text{G}_{t'+1}$, which contradicts the fact that $j$ is discarded. Then, notice that $S^*\cap U^t$ is an independent subset of $U^t$. Hence, its cardinality is no more than $t$ according to the above claim. The partition $U^t = \cup^t_{i=0} (U^i\setminus U^{i-1})$ gives us that $\sum^t_{i=1}s_{i-1} = |S^*\cap U^t|\leq t.$
\hfill\QEDA\end{proof}

We use (\ref{eq: lemmainequality2}) to obtain an upper bound to the right-hand side of (\ref{eq: lemmainequality1}) and consequently to derive an upper bound of $f(S^*)-f(S^\text{G})$. The following explains in detail these steps to prove the performance guarantee in (\ref{eq: performanceguaranteeMatroidOpt}).

\begin{proof}[Proof of Theorem 2] First, we consider the case in which $\rho_i$, $i = 0,\ldots, K-1$, are distinct. We define $t_1$ such that $\rho_{t_1-1}$ is the largest among $\rho_0, \rho_1,\ldots,\rho_{K-1}$ and $t_2$ such that $\rho_{t_2-1}$ is the largest among $\rho_{t_1}, \rho_{t_1+1},\ldots,\rho_{K-1}$. Following the same pattern we have $t_1,t_2,\ldots,t_p,$ where $t_p = K$. Since $s_{i}\geq 0$ is bounded by (\ref{eq: lemmainequality2}), to give an upper bound to the right-hand side of (\ref{eq: lemmainequality1}), we construct a linear program as follows,
\begin{equation}
\begin{aligned}
& \max_{s_1,s_2,\ldots,s_{K}}
& & \sum_{i=1}^{K} \rho_{i-1} s_{i-1}
\\ 
&\ \ \ \ \ \mathrm{s.t.} & & \sum_{i=1}^{t}s_{i-1} \leq t , \text{ } t = 1,2,\ldots,K,\\
& & & s_{t-1}\geq 0,\text{ }t = 1,2,\ldots,K.
\end{aligned}
\label{eq:matroidconstrained optimization transferred}
\end{equation}

Let $s_{i-1}^*$, $i = 1,2,\ldots,K,$ denote the optimal solution. We claim $s_{t_1-1}^* = t_1$. Otherwise, $s_{t_1-1}^* < t_1$ and due to (\ref{eq: lemmainequality2}) two situations might happen, \textit{a)} $\sum^{t_1}_{i=1} s^*_{i-1} = t_1$ or \textit{b)} $\sum^{t_1}_{i=1} s^*_{i-1} < t_1$.

For case \textit{a)}, we obtain $\sum^{t_1-1}_{i=1} s^*_{i-1} > 0$. It follows that there exists $l < t_1\text{ such that } s^*_{l-1}>0$. Then, we decrease $s^*_{l-1}$ by $\delta>0$ and increase $s^*_{t_1-1}$ also by $\delta$. The value of $\delta$ is small enough so that $s^*_{l-1}>0$. This operation decreases $\sum^{t}_{i=1} s^*_{i-1}$ for $l\leq t\leq t_1-1$ and keeps the sum unchanged for any other $t$, so the constraints of (\ref{eq:matroidconstrained optimization transferred}) are not violated. Also considering that $\rho^*_{t_1-1}>\rho^*_{l-1}$, after these changes, the objective function is strictly greater than the value obtained at the original optimum. Thus, case \textit{a)} is impossible. 

For case \textit{b)}, we collect all the integers $l>t_1$ satisfying $s^*_{l-1}>0$. Assume they are $l_q>\ldots>l_1>t_1$. We have $q\geq 1$. Otherwise, $s^*_{l-1}=0$ for any $l>t_1$ and we can increase $s^*_{t_1-1}$ by a small amount to obtain a greater value of the objective function without violating the constraints. Knowing that $s^*_{l_1-1}>0$ and following the same reasoning provided in the discussion for the case \textit{a)}, we increase $s^*_{t_1-1}$ and decrease $s^*_{l_1-1}$ with the same amount. This way, an objective value is obtained larger than that evaluated at the original optimum. Thus, case \textit{b)} is impossible. 

In conclusion, $s_{t_1-1}^* = t_1$ and (\ref{eq:matroidconstrained optimization transferred}) is equivalent to 
\begin{equation}
\begin{aligned}
& \max_{s_{t_1},\ldots,s_{K-1}}
& & \sum_{i=t_1+1}^{K} \rho_{i-1} s_{i-1}
\\ 
&\ \ \ \ \mathrm{s.t.} & & \hspace{-.2cm}\sum_{i=t_1+1}^{t}s_{i-1} \leq t-t_{1} , \text{ } t = t_1+1,\ldots,K,\\
& & & s_{t-1}\geq 0,\text{ }t = t_1+1,\ldots,K.
\end{aligned}
\label{eq:matroidconstrained optimization transferred2}
\end{equation}
We determine $s^*_{t_2-1}$ in the same way as we determine $s^*_{t_1-1}$ in (\ref{eq:matroidconstrained optimization transferred}). By repeating the above procedure we obtain the solution to (\ref{eq:matroidconstrained optimization transferred}) as 
  \begin{equation}
    s_{i-1}^*=
    \begin{cases}
      t_1, & \text{if } i=t_1,\\
      t_{j}-t_{j-1}, & \text{if } i=t_j\text{ and $j\neq 1$},\\
    0, & \text{otherwise}.
    \end{cases}
    \label{eq:soptimal}
  \end{equation}
  
Now, if $\rho_{i}$, $i=0,\ldots,K-1$ are not distinct, that is, there exist $i_1<i_2<\cdots<i_q$ such that $\rho_{i_1} = \rho_{i_2} = \cdots = \rho_{i_q}$, we can let $s^*_{i_1} = s^*_{i_2} = \cdots =s^*_{i_{q-1}} = 0 $ and obtain the same solution as (\ref{eq:soptimal}).

Next, notice that
\begin{equation}
 \begin{aligned}
    \rho_{i_2} & = f(S^{i_2+1})-f(S^{i_2})\\
    &\leq \gamma^{-1}(f(S^{i_1}\cup\{v^\text{G}_{i_2+1}\})-f(S^{i_1}))\\
    &\leq \gamma^{-1}\rho_{i_1},
    \end{aligned}
    \label{eq:rhoi1rhoi2}
  \end{equation}
where the first inequality comes from the definition of submodularity ratio, while the second is due to (\ref{eq: greedypickproperty}).
Substituting the optimal solution into the objective function and considering (\ref{eq:rhoi1rhoi2}), we have 
 \begin{equation}
 \begin{aligned}
    \sum_{i=1}^{K} \rho_{i-1} s^*_{i-1} &= t_1\rho_{t_1-1}+\cdots+(t_p-t_{p-1})\rho_{t_{p}-1}\\
    &\leq \gamma^{-1}\sum^p_{k=1}\sum_{i=t_{k-1}+1}^{t_k}\rho_{i-1}\\
    &=\gamma^{-1}\sum^{K}_{i = 1}\rho_{i-1}\\
    &= \gamma^{-1}(f(S^\text{G})-F(\emptyset)).
    \end{aligned}
    \label{eq:backwardInequalities}
  \end{equation}
Combining (\ref{eq: ideaTheorem}), (\ref{eq: lemmainequality1}) and (\ref{eq:backwardInequalities}), we have
   \begin{equation*}
 \begin{aligned}
    f(S^*)-f(S^\text{G}) &\leq \gamma^{-1}\sum_{j\in S^*\setminus S^\text{G}}\rho_j(S^\text{G})\\
    &\leq \gamma^{-2} \sum_{i=1}^{K} \rho_{i-1}s_{i-1}^* \\
    &\leq \gamma^{-3} \Big(f(S^\text{G})-f(\emptyset)\Big).
    \end{aligned}
    \label{eq: totalinequalityforMatroidOptimization}
  \end{equation*}
By rewriting the above inequality, we have 
$$\frac{f(S^\text{G})-f(\emptyset)}{f(S^*)-f(\emptyset)}\geq\frac{\gamma^3}{\gamma^3+1}.$$\hfill\QEDA
\end{proof}

Notice that this guarantee coincides with that in \cite{fisher1978analysis} if $\gamma = 1$. We also refer to Appendix for a comparison of this guarantee with the one given by \cite{pmlr-v80-chen18b} using the residual random greedy algorithm. Since in the proof of Theorem~\ref{thm:upperlimitFormatroidconstraints}, only (\ref{eq: ideaTheorem}), (\ref{eq: submodularityratioTheorem}) and (\ref{eq:rhoi1rhoi2}) utilize $\gamma$ to form the inequalities, we denote the maximum $\gamma$ satisfying these three inequalities as $\gamma_G$. If $\gamma$ is replaced by $\gamma_G$, the performance guarantee in (\ref{eq: performanceguaranteeMatroidOpt}) would still hold. We call $\gamma_G$ the greedy submodularity ratio. Clearly, $\gamma_G$ is easier to compute than $\gamma$ and gives a better performance guarantee because $\gamma_G\geq \gamma$. One can calculate $\gamma_G$ only after obtaining $S^\text{G}$ by analyzing the three inequalities. The calculation takes $O(n^K)$ steps. Notice that the value of $\gamma_G$ is influenced by the constraint set. If the constraint changes, the inequalities defining $\gamma_G$ are different, and as a result $\gamma_G$ changes. In contrast, submodularity ratio $\gamma$ only depends on the objective function.

By substituting $f=-F$ and $\gamma_G$, the greedy submodularity ratio of $f$, into the performance guarantee (\ref{eq: performanceguaranteeMatroidOpt}), we obtain
   \begin{equation}
  \frac{F(\emptyset)-F(S^\text{G})}{F(\emptyset)-F(S^*)}\geq\frac{\gamma_G^3}{\gamma_G^3+1}.
    \label{eq: boundwithEMPTYSET}
  \end{equation}

\section{Implementation and Numerical Results}
\label{sec:num}
\subsection{Feasibility Check over the Matroid}
\label{sec:feasibilitycheck}
To implement the greedy algorithm on Problem (\ref{eq:mainproblem}), we need a method to determine whether it is feasible to add $\omega\in V\setminus S^k$ to $S^k$ to form $S^{k+1}\in \mathcal{F}$ for $k< K$. To this end, we show that this feasibility check can be cast as a maximum matching problem in a bipartite graph derived from the original graph $(V,E)$. With this equivalence, we can use standard algorithms for maximum matchings in bipartite graphs to check feasibility.

To start with, we introduce the concept of matching in bipartite graphs. An undirected graph is called bipartite and denoted as $(V^1, V^2, \mathsf{E})$ if its vertices are partitioned into $V^1$ and $V^2$ while any undirected edge in $\mathsf{E}$ connects a vertex in $V^1$ to another in $V^2$. A matching $m$ is a subset of $\mathsf{E}$ if no two edges in $m$ share a vertex in common. Given a subset $L$ of $V^1\cup V^2$, we say $L$ is covered by $m$ if any $v\in L$ is adjacent to an edge in $m$. Matching $m$ is maximum if it has the largest cardinality among all the matchings and is perfect if $V^2$ is covered.

Given the graph $(V,E)$ describing System (\ref{eq: systemmodel}), we build an auxiliary bipartite graph to determine whether a given actuator set $S$ is capable. $V'=\{v'_1,\ldots, v'_n\}$ is built as a copy of $V =\{{v}_1,\ldots, {v}_n\}$ and $S'\subset V'$ denotes the copy of $S$. As for the auxiliary edges, we have $\mathsf{E}$ consisting of undirected edges connecting $v_{i}$ with $v'_{j}$ if $(v_i,v_j)\in E$ and $\mathsf{E}_S\subset \mathsf{E}$ consisting of undirected edges adjacent to $v'_k$ for any $k$ such that $v_k \in S$.
The bipartite graph we need to check to determine whether $S$ is capable is defined by a function mapping $S$ to subgraphs of $(V,V',\mathsf{E})$, specifically $\mathcal{H}_b(S) = (V, V'\setminus S', \mathsf{E}\setminus \mathsf{E}_S)$. We propose our method for feasibility check by showing the following theorem.

\begin{theorem}
\label{thm: feasibility}
Given the cardinality limit $K$ and an actuator set $S$ with $|S|=k< K$, $S\in \mathcal{\tilde{C}}_K$ if and only if $|\bar{m}(S)| \geq n-K$, where $\bar{m}(S)$ is a maximum matching in $\mathcal{H}_b(S)$.
\end{theorem}

\begin{proof}
Recall that $S$ is a capable actuator set if and only if $S$ is a leader selection that achieves structural controllability. According to \cite{clark2012leader}, the latter statement is true if and only if $S\neq \emptyset$ and there exists a perfect matching in $\mathcal{H}_b(S)$. We use this graphical property to check structural controllability in the proof.

``$\Rightarrow$": If $S\in \mathcal{\tilde{C}}_K$, there exist $Q \in \mathcal{C}_K$ such that $S\subset Q$ and thus a perfect matching $m$ in $\mathcal{H}_b(Q)$. By the definition of perfect matching, $|m|= n-K$. Since $\mathcal{H}_b(Q)$ is a subgraph of $\mathcal{H}_b(S)$, $m$ is also a matching in $\mathcal{H}_b(S)$. Thus $|\bar{m}(S)|\geq |m| = n-K$.

``$\Leftarrow$": We pick a maximum matching in $\mathcal{H}_b(S)$ and denote it as $m^*$. Suppose in $\mathcal{H}_b(S)$, $P'=\{v'_{i_1},\ldots,v'_{i_d}\}$ is the largest subset of $V'\setminus S'$ missed by $m^*$. Since in $V'\setminus S'$ there are at least $n-K$ vertices covered by $m^*$, we have $d \leq K-k$. Let $Q = P\cup S$, where $P=\{v_{i_1},\ldots,v_{i_d}\}$. Matching $m^*$ is perfect in $\mathcal{H}_b(Q)$ because no vertex in $V'\setminus (P'\cup S')$ is missed by $m^*$. Hence, $Q$ is a capable actuator set. Also considering $|Q|\leq |J|+|S|\leq K$, we have $Q\in \mathcal{\tilde{C}}_K$. Since $S\subset Q$, we also have $S\in \mathcal{\tilde{C}}_K$. \hfill\QEDA
\end{proof}

To obtain a maximum matching in $\mathcal{H}_b(S)$, it is known that one can use a max-flow algorithm \cite{plummer1986matching}. Here we adopt Edmond-Karp algorithm to calculate the maximum flow \cite{Edmonds:1972:TIA:321694.321699}. This requires $O(pq^2)$ steps, where $p$ and $q$ respectively denote node cardinality and edge cardinality in the flow graph generated based on $\mathcal{H}_b(S)$. For example, in $\mathcal{H}(\emptyset)$, $p = 2n+2$ and $q = 2n+|E|$. To conclude, we examine whether $\{\omega\}\cup S^t$ belongs to $\mathcal{\tilde{C}}_K$ by using Edmond-Karp algorithm to calculate the maximum matching cardinality in $\mathcal{H}_b(\{\omega\}\cup S^k)$. The node $\omega$ passes the feasibility check if the matching consists of at least $n-K$ edges.

\subsection{Example on a 4-Node Network}
Suppose a system has 4 nodes, described by the dynamic equations \eqref{eq: systemmodel} where
$$A = \begin{bmatrix}
   0 & -0.5 & -0.8 & -0.6\\
   1 & 0 & 0 & 0\\
   1 & 0 & 0 & 0\\
   1 & 0 & 0 & 0
\end{bmatrix}.$$
The digraph $(V,E)$ corresponding to this system is provided in Figure \ref{fig:oaa}. To calculate the metric $F(S)$ in \eqref{eq:F(S)}, we let $T=2$ and $\epsilon = 10^{-9}$. When applying the greedy algorithm on this example, one starts from $S=\emptyset$ and builds the auxiliary bipartite graph $\mathcal{H}_b(\emptyset)$. It is easy to see any maximum matching consists of two edges, that is, $\bar{m}(\emptyset) = 2$. According to Theorem \ref{thm: feasibility}, for $\emptyset$ to belong to $\tilde{\mathcal{C}}_K$, we require $K\geq 2$. It follows that Problem (\ref{eq:mainproblem}) is feasible only when $K\geq 2$.
\begin{figure}
    \centering
    \includegraphics[width=0.7\linewidth]{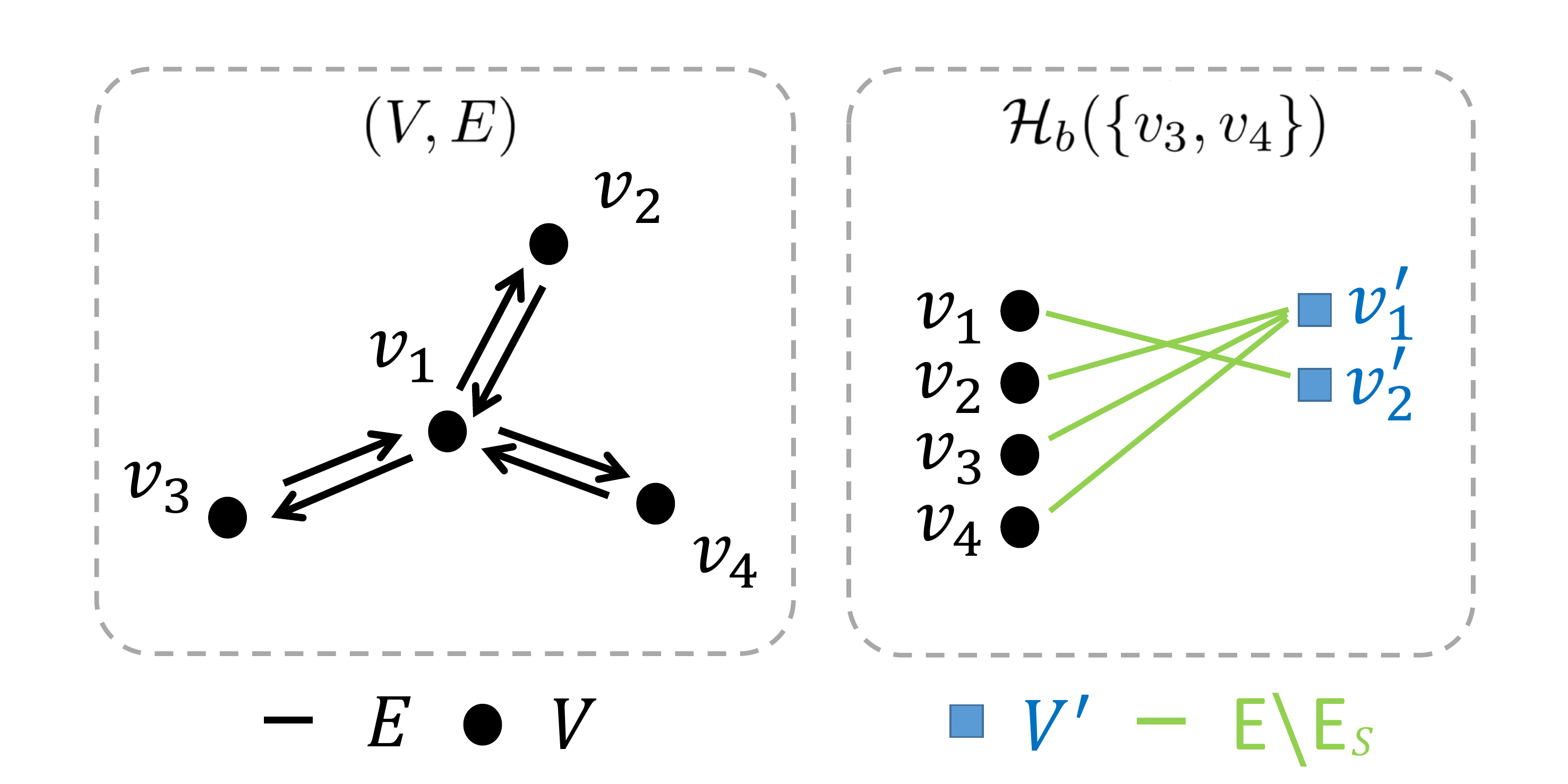}
    \caption{Original graph $(V,E)$ and the auxiliary graph $\mathcal{H}_b(\{v_3,v_4\})$}
    \label{fig:oaa}
    \vspace{.5cm}
\end{figure}

Suppose at most two input nodes are allowed, that is, $K = 2$. The greedy algorithm first examines $v_3$, because $F(\{v_3\})$ is the smallest among $F(\{v_i\})$, $i=1,2,3,4$. Notice that in $\mathcal{H}_b(\{v_3\})$ there exists a matching consisting of two edges. Since $n-K=2$, we know from Theorem \ref{thm: feasibility} that $\{v_3\}$ belongs to $\tilde{\mathcal{C}}_K$. Hence, the first node selected is $v_3$. Notice that even if $F(\{v_1\})$ were the smallest, $v_1$ would not pass the feasibility check, because any maximum matching in $\mathcal{H}_b(\{v_1\})$ only contains one edge. Thus, $\bar{m}(\{v_1\})< n-K$. This implies that $v_1$ does not belong to any capable actuator set with 2 elements. Then, the second node selected is ${v_4}$. Thus, $S^\text{G}=\{v_3,v_4\}$. We illustrate the bipartite graph $\mathcal{H}_b(\{v_3,v_4\})$ in Figure \ref{fig:oaa}. It can be seen that there exists a perfect matching. Consequently, $S^\text{G}$ is indeed a capable actuator set.

\subsection{Experiment on a 23-Node Network}
\label{sec:experiment}
\begin{figure}
    \centering
    \includegraphics[width=0.8\linewidth]{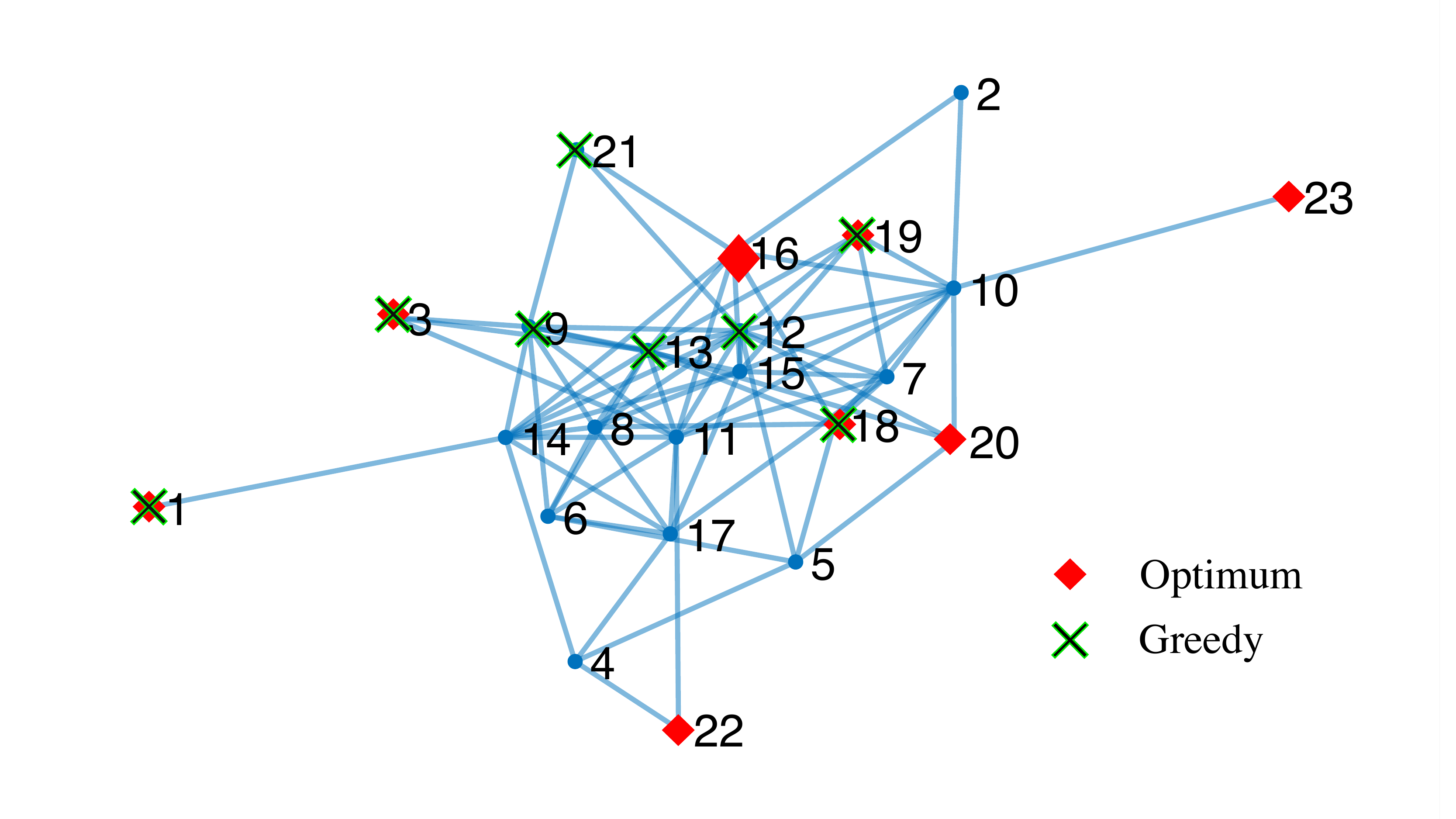}
    \caption{Greedy selection versus the optimal}
    \label{fig:experimentresult}
    \vspace{.5cm}
\end{figure}
We utilize Algorithm \ref{alg:ALG4} on an undirected unweighted graph. To gain insight into how the selections by $S^\text{G}$ and $S^*$ respectively depend on node connectivity, we assign different degrees to different vertices and generate the graph shown in Figure \ref{fig:experimentresult} using the toolbox \cite{NetworkToolbox}. Specifically, vertex $i$ has a degree of $i$ if $i<13$ and a degree of $24-i$ if $i\geq12$. To calculate the objective function $F(S)$ in \eqref{eq:F(S)}, we let $\epsilon = 10^{-9}$ and $T = 2$.  

Illustrated in Figure \ref{fig:experimentresult}, the greedy selection includes some vertices with high degrees that are avoided by the optimal selection. The metric values are $F(S^\text{G})=$ $15976.0$ versus $F(S^*)=$ $6043.8$. The greedy algorithm picks the nodes in the order of $3$, $9$, $13$, $12$, $19$, $21$, $1$, $18$, so it tends to pick some high-degree ones at early stages. In contrast, most late-stage picks feature low degrees and intersect with the optimum $S^*$. This is because the result of the marginal gain analysis varies if different sets are based on. A high-degree node preferred by the greedy algorithm at early stages produces at late stages less marginal gain than the low-degree ones. Thus, the nodes with the largest marginal gains collected from every stage may not perform well as a whole. This shows the disadvantage of the greedy algorithm.

It is numerically verified that $\gamma_G = 1$ satisfies (\ref{eq: ideaTheorem}), (\ref{eq: submodularityratioTheorem}) and (\ref{eq:rhoi1rhoi2}) under the setting of this experiment. Thus, $\gamma_G = 1$ is the greedy submodularity ratio. Let
$$\alpha = \frac{F(\emptyset)-F(S^\text{G})}{F(\emptyset)-F(S^*)}.$$
Substituting $F(S^\text{G})$ and $F(S^*)$, we have $\alpha = 1 - 4.3\times10^{-7}$. By the guarantee in (\ref{eq: boundwithEMPTYSET}), we can ensure without calculation of $S^*$ that $\alpha\geq 0.5$. But due to the fact that $F(\emptyset)$ is large, the bound for $F(S^\text{G})$ resulting from the guarantee in (\ref{eq: boundwithEMPTYSET}) is loose.

\section{Conclusions}

In this paper, we aimed to pick an actuator set to minimize a controllability metric based on average energy consumption while ensuring that the system is structurally controllable. We showed that this problem can be reformulated as a weakly submodular optimization problem over matroid constraints. Given the submodularity ratio of the objective function, we bounded the worst-case performance of the greedy algorithm applied to this class of problems. To implement the greedy algorithm, we proved that the feasibility check over the structural controllability matroid can be done via calculating a maximum matching on a certain auxiliary bipartite graph resulting from the network graph.

Our future work is focused on exploring the tightness of our performance guarantee. Inspired by the numerical observations, we also aim to investigate the network structures under which the energy-based controllability metric would be submodular. 
\bibliographystyle{IEEEtran}
\bibliography{IEEEabrv,library}

\begin{thebibliography}{10}
\providecommand{\url}[1]{#1}
\csname url@samestyle\endcsname
\providecommand{\newblock}{\relax}
\providecommand{\bibinfo}[2]{#2}
\providecommand{\BIBentrySTDinterwordspacing}{\spaceskip=0pt\relax}
\providecommand{\BIBentryALTinterwordstretchfactor}{4}
\providecommand{\BIBentryALTinterwordspacing}{\spaceskip=\fontdimen2\font plus
\BIBentryALTinterwordstretchfactor\fontdimen3\font minus
  \fontdimen4\font\relax}
\providecommand{\BIBforeignlanguage}[2]{{%
\expandafter\ifx\csname l@#1\endcsname\relax
\typeout{** WARNING: IEEEtran.bst: No hyphenation pattern has been}%
\typeout{** loaded for the language `#1'. Using the pattern for}%
\typeout{** the default language instead.}%
\else
\language=\csname l@#1\endcsname
\fi
#2}}
\providecommand{\BIBdecl}{\relax}
\BIBdecl

\bibitem{MULLER1972237}
P.~M{\"u}ller and H.~Weber, ``Analysis and optimization of certain qualities of
  controllability and observability for linear dynamical systems,''
  \emph{Automatica}, vol.~8, no.~3, pp. 237--246, 1972.

\bibitem{Redmond1996}
J.~Redmond and G.~Parker, ``Actuator placement based on reachable set
  optimization for expected disturbance,'' \emph{Journal of Optimization Theory
  and Applications}, vol.~90, no.~2, pp. 279--300, Aug 1996.

\bibitem{Ali2017MinimalReachablity}
A.~Jadbabaie, A.~Olshevsky, G.~J. Pappas, and V.~Tzoumas, ``Minimal
  reachability is hard to approximate,'' \emph{IEEE Transactions on Automatic
  Control}, vol.~64, no.~2, pp. 783--789, 2019.

\bibitem{SubmodularityandControl}
T.~H. Summers, F.~L. Cortesi, and J.~Lygeros, ``On submodularity and
  controllability in complex dynamical networks,'' \emph{IEEE Transactions on
  Control of Network Systems}, vol.~3, no.~1, pp. 91--101, March 2016.

\bibitem{nemhauser1978analysis}
G.~L. Nemhauser, L.~A. Wolsey, and M.~L. Fisher, ``An analysis of
  approximations for maximizing submodular set functions-{I},''
  \emph{Mathematical programming}, vol.~14, no.~1, pp. 265--294, 1978.

\bibitem{summers2017performance}
T.~Summers and M.~Kamgarpour, ``Performance guarantees for greedy maximization
  of non-submodular controllability metrics,'' \emph{arXiv preprint
  arXiv:1712.04122}, 2017.

\bibitem{tyler2018correction}
T.~H. Summers, F.~L. Cortesi, and J.~Lygeros, ``Corrections to “on
  submodularity and controllability in complex dynamical networks”,''
  \emph{IEEE Transactions on Control of Network Systems}, vol.~5, no.~3, pp.
  1503--1503, Sept 2018.

\bibitem{Das:2011:SMS:3104482.3104615}
A.~Das and D.~Kempe, ``Submodular meets spectral: greedy algorithms for subset
  selection, sparse approximation and dictionary selection,'' in
  \emph{Proceedings of the 28th International Conference on International
  Conference on Machine Learning}, 2011, pp. 1057--1064.

\bibitem{bian2017guarantees}
A.~A. Bian, J.~M. Buhmann, A.~Krause, and S.~Tschiatschek, ``Guarantees for
  greedy maximization of non-submodular functions with applications,'' in
  \emph{Proceedings of the 34th International Conference on Machine Learning},
  2017, pp. 498--507.

\bibitem{SturcturalcontrollabilityTai}
C.-T. Lin, ``Structural controllability,'' \emph{IEEE Transactions on Automatic
  Control}, vol.~19, no.~3, pp. 201--208, June 1974.

\bibitem{liu2011controllability}
Y.-Y. Liu, J.-J. Slotine, and A.-L. Barab{\'a}si, ``Controllability of complex
  networks,'' \emph{{N}ature}, vol. 473, no. 7346, p. 167, 2011.

\bibitem{clark2012leader}
A.~Clark, L.~Bushnell, and R.~Poovendran, ``On leader selection for performance
  and controllability in multi-agent systems,'' in \emph{51st IEEE Conference
  on Decision and Control (CDC)}, 2012, pp. 86--93.

\bibitem{fisher1978analysis}
M.~L. Fisher, G.~L. Nemhauser, and L.~A. Wolsey, ``An analysis of
  approximations for maximizing submodular set functions-{II},'' in
  \emph{Polyhedral combinatorics}.\hskip 1em plus 0.5em minus 0.4em\relax
  Springer, 1978, pp. 73--87.

\bibitem{pmlr-v80-chen18b}
L.~Chen, M.~Feldman, and A.~Karbasi, ``Weakly submodular maximization beyond
  cardinality constraints: Does randomization help greedy?'' in
  \emph{Proceedings of the 35th International Conference on Machine Learning},
  2018, pp. 804--813.

\bibitem{EffortBounds}
V.~Tzoumas, M.~A. Rahimian, G.~J. Pappas, and A.~Jadbabaie, ``Minimal actuator
  placement with bounds on control effort,'' \emph{IEEE Transactions on Control
  of Network Systems}, vol.~3, no.~1, pp. 67--78, March 2016.

\bibitem{patterson2010leader}
S.~Patterson and B.~Bamieh, ``Leader selection for optimal network coherence,''
  in \emph{49th IEEE Conference on Decision and Control (CDC)}, 2010, pp.
  2692--2697.

\bibitem{plummer1986matching}
L.~Lov{\'a}sz and M.~Plummer, \emph{Matching Theory}, ser. North-Holland
  Mathematics Studies.\hskip 1em plus 0.5em minus 0.4em\relax Elsevier Science,
  1986.

\bibitem{Edmonds:1972:TIA:321694.321699}
J.~Edmonds and R.~M. Karp, ``Theoretical improvements in algorithmic efficiency
  for network flow problems,'' \emph{Journal of the ACM}, vol.~19, no.~2, pp.
  248--264, 1972.

\bibitem{NetworkToolbox}
G.~Bounova, ``Octave networks toolbox,'' 2015, doi: \url{10.5281/zenodo.22398}.

\end{thebibliography}
\vspace{.5cm}

\section{Appendix}
\subsection{Definitions of submodularity ratio}

Let $\gamma_1$ denote the submodularity ratio of $f$ from Definition~\ref{def:submodularityratio}.
It is straightforward to see that $\gamma = \gamma_1$ satisfies
\begin{equation}
\label{eq:definition for submodularity ratio2}
\gamma\rho_{U}(S) \leq \sum_{\omega\in U\setminus S}\rho_\omega(S), \forall S,U \subset V.
\end{equation}
However, the largest $\gamma$ satisfying the above set of inequalities, denoted as $\gamma_2$, does not necessarily satisfy~\eqref{eq:definition for submodularity ratio} given in Definition~\ref{def:submodularityratio}. Hence, we have $\gamma_2\geq\gamma_1$.

There are previous studies in the literature defining the submodularity ratio as $\gamma_2$ instead of $\gamma_1$~\cite{bian2017guarantees,summers2017performance,pmlr-v80-chen18b}. In the proof of Theorem \ref{thm:upperlimitFormatroidconstraints}, as we are deriving (\ref{eq: submodularityratioTheorem}), we use the inequalities \eqref{eq:definition for submodularity ratio} from Definition~\ref{def:submodularityratio}. One can verify that the inequalities in (\ref{eq:definition for submodularity ratio2}) would not allow us to derive (\ref{eq: submodularityratioTheorem}). Hence, the performance guarantee (\ref{eq: performanceguaranteeMatroidOpt}) does not extend to the submodularity ratio $\gamma_2$. 

 The work in \cite{bian2017guarantees} provides a performance guarantee for the greedy algorithm applied to weakly submodular optimization involving cardinality constraints. This guarantee improves with increasing $\gamma_2$. Since we have $\gamma_2\geq\gamma_1$, the guarantee also holds if $\gamma_2$ is replaced by $\gamma_1$. In addition, the work of \cite{summers2017performance} obtains a lower bound for $\gamma_2$ for the metric $-F$ in \eqref{eq:F(S)} based on eigenvalue inequalities for sum and product of matrices. One can easily verify that this lower bound is also applicable to $\gamma_1$ from Definition~\ref{def:submodularityratio}.

To the best of our knowledge, the guarantee in (\ref{eq: performanceguaranteeMatroidOpt}) is the first performance guarantee for the greedy algorithm applied to matroid optimization problems featuring weakly submodular objective functions. The work of \cite{pmlr-v80-chen18b} exploited the submodularity ratio defined by (\ref{eq:definition for submodularity ratio2}) and obtained a guarantee for the residual random greedy algorithm on the same problem. We denote the final set returned by this algorithm as $S^{\text{RRG}}$. The guarantee provided in \cite{pmlr-v80-chen18b} for this class of randomized algorithms is
   \begin{equation}
    \frac{f(S^{\text{RRG}})-f(\emptyset)}{f(S^*)-f(\emptyset)}\geq \frac{\gamma_2^2}{(1+\gamma_2)^2}.
    \label{eq: RRGguarantee}
  \end{equation}
Let $\gamma$ denote the theoretical lower bound derived in \cite{summers2017performance} for the metric $-F$ in \eqref{eq:F(S)}. This lower bound satisfies $\gamma_2\geq\gamma_1\geq\gamma$. Since $\gamma$ is applicable to both \eqref{eq: performanceguaranteeMatroidOpt} and \eqref{eq: RRGguarantee}, we let $a_1(\gamma) = \gamma^3/(1+\gamma^3)$ and $a_2(\gamma) = \gamma^2/(1+\gamma)^2$ denote the theoretical guarantees associated with (\ref{eq: performanceguaranteeMatroidOpt}) and \eqref{eq: RRGguarantee}, respectively. These two functions are plotted in Figure \ref{fig:comparisonbetweentwoguarantees}. We see that the guarantee we derived in (\ref{eq: performanceguaranteeMatroidOpt}) is tighter than the one from \cite{pmlr-v80-chen18b}, if the lower bound $\gamma>0.5$.
\begin{figure}
    \centering
    \includegraphics[width=0.7\linewidth]{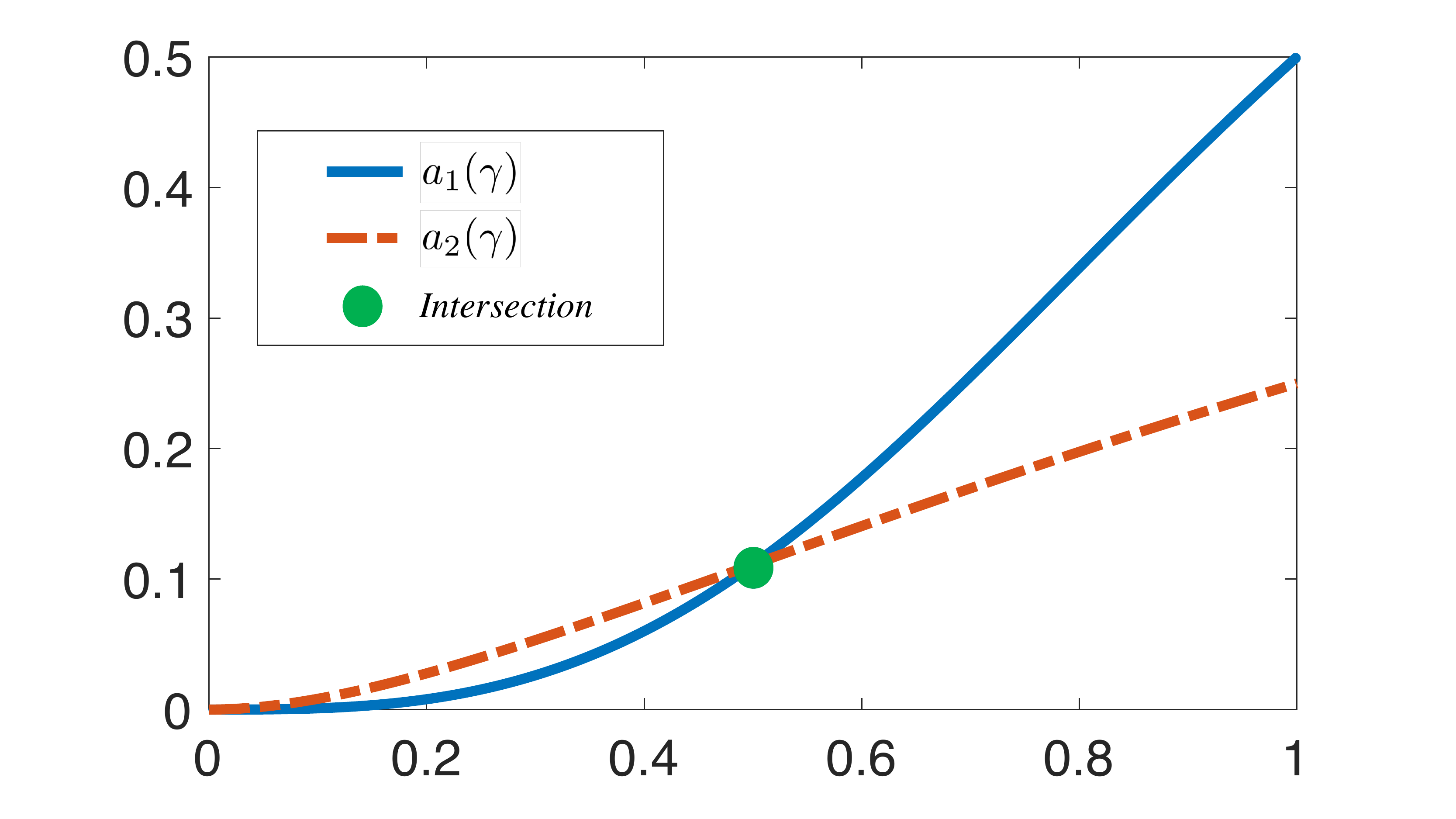}
    \caption{A comparison between two guarantees}
    \label{fig:comparisonbetweentwoguarantees}
    \vspace{.5cm}
\end{figure}

\end{document}